\documentclass[twoside, 11pt]{article}

\usepackage[a4paper, left = 1in, right = 1in, textheight = 22cm]{geometry}

\usepackage{enumerate}

\usepackage{amsmath, amsfonts, amssymb, mathtools, amsthm}
\usepackage{dsfont}

\usepackage{hyperref}
\hypersetup{breaklinks=true}
\usepackage{cleveref}

\usepackage{xurl}
\usepackage{biblatex}

\DeclareFieldFormat{doi}{%
  \hfil\penalty90\hfilneg\space DOI\addcolon\addnbspace
  \ifhyperref
    {\href{https://doi.org/#1}{\nolinkurl{#1}}}
    {\nolinkurl{#1}}}

\DeclarePairedDelimiter\abs{\lvert}{\rvert}
\DeclareMathOperator{\Span}{span}
\DeclareMathOperator{\supp}{supp}
\DeclareMathOperator{\cl}{cl}
\DeclareMathOperator{\card}{card}

\renewcommand{\leq}{\leqslant}
\renewcommand{\geq}{\geqslant}
\renewcommand{\subsetneq}{\varsubsetneq}

\newcommand{\Q}{\mathbb{Q}}
\newcommand{\C}{\mathbb{C}}

\newcommand{\N}{\mathbb{N}}

\newcommand{\D}{\mathbb{D}}

\theoremstyle{plain}
\newtheorem{theorem}{Theorem}[section]
\newtheorem{proposition}[theorem]{Proposition}
\newtheorem{lemma}[theorem]{Lemma}

\theoremstyle{definition}
\newtheorem{definition}[theorem]{Definition}

\theoremstyle{remark}
\newtheorem{remark}[theorem]{Remark}

\crefname{proposition}{Proposition}{Propositions}
\crefname{lemma}{Lemma}{Lemmata}
\crefname{theorem}{Theorem}{Theorems}
\crefname{equation}{}{}
\crefname{remark}{Remark}{Remarks}

\makeatletter
\let\oldabs\abs
\def\abs{\@ifstar{\oldabs}{\oldabs*}}

\nocite{*}

\title{\textbf{Uncountably infinite algebraic genericity and spaceability for sequence spaces}}
\author{C. A. Konidas}
\date{\vspace{-5ex}}

\begin{document}
\pagestyle{myheadings}
\markboth{Uncountably infinite algebraic genericity and spaceability for sequence spaces}{C. A. Konidas}
\maketitle
\begin{abstract}
\noindent Let $X$ be a topological vector space of complex-valued sequences and $Y$ be a subset of $X$.
We provide conditions for $X \setminus Y \cup \{0\}$ to contain uncountably infinitely many linearly independent dense vector subspaces of $X$.
We also provide conditions for $X \setminus Y \cup \{0\}$ to contain uncountably infinitely many linearly independent closed infinite-dimensional vector subspaces of $X$.
We apply these results to a chain of spaces containing the $\ell^p$ spaces.
\end{abstract}
{\em AMS classification numbers}: 15A03, 46A45, 46A11, 46A13\smallskip

\section{Introduction}

Let $E$ be a topological vector space and $M$ be a subset of $E$.
\begin{itemize}

\item We say that $M$ is \emph{algebraically generic} if and only if there exists a dense vector subspace of $E$ contained in the set $M \cup \{0\}$. In particular, we say that $M$ is \emph{maximal algebraically generic} if and only if there exists a dense vector subspace of $E$ contained in the set $M \cup \{0\}$ with dimension equal to the dimension of $E$.

\item We say that $M$ is \emph{spaceable} if and only if there exists a closed infinite-dimensional vector subspace of $E$ contained in the set $M \cup \{0\}$.

\end{itemize}
The notions of algebraic genericity and spaceability are both concerned with finding a vector subspace of $E$ contained in the set $M \cup \{0\}$ satisfying a certain property.
In this paper we extend these notions by searching for multiple subspaces of $E$ contained in the set $M \cup \{0\}$ that are linearly independent and satisfy the respective property.

A family $(F^k)_{k \in I}$ of subspaces of a vector space is said to be linearly independent if and only if for every $J$ finite subset of $I$ and $v_{k} \in F^k$ for $k \in J$ such that $\sum_{k \in J}{v_j} = 0$ it follows that $v_k = 0$ for every $k \in J$.

We consider the chain
\begin{equation*}
\label{chain}
c_{00} \subsetneq A^{\infty}(\D) \subsetneq \bigcap_{p>0}{\ell^p} \subsetneq \ell^a \subsetneq \bigcap_{q>a}{\ell^q} \subsetneq \ell^b \subsetneq \bigcap_{p>b}{\ell^p} \subsetneq c_0 \subsetneq \ell^{\infty} \subsetneq H(\D) \subsetneq \C^{\N_0} \tag{$\star$}
\end{equation*}
of sequence spaces where $a,b \in (0,\infty)$ with $a < b$.
For completeness we provide the definitions of these spaces and the metric with which we consider each space in the next section.
This chain has been studied in \cite{BIG}, while a smaller version of it has been studied in \cite{NEST, BERNAL}.
In particular, the following two theorems were proven in \cite{BIG}.

\begin{theorem}
Let $Y,X$ be spaces of the chain \cref{chain} such that $Y \subsetneq X$ and $X \neq \ell^{\infty}$.
Then there exists a dense vector subspace of $X$ contained in the set $X \setminus Y \cup \{0\}$, that is, $X \setminus Y$ is algebraically generic in $X$.
\end{theorem}

\begin{theorem}
Let $Y,X$ be spaces of the chain \cref{chain} such that $Y \subsetneq X$.
Then there exists a closed infinite-dimensional vector subspace of $X$ contained in the set $X \setminus Y \cup \{0\}$, that is, $X \setminus Y$ is spaceable in $X$.
\end{theorem}

We improve these results by proving the two theorems that follow, where $\mathfrak{c}$ denotes the cardinality of the continuum.

\begin{theorem}
Let $Y,X$ be spaces of the chain \cref{chain} such that $Y \subsetneq X$ and $X \neq \ell^{\infty}$.
Then there exist a set $I$ with $\card(I) = \mathfrak{c}$ and a linearly independent family $(F^k)_{k \in I}$ of dense vector subspaces of $X$ such that the vector subspace generated by the set $\bigcup_{k \in I}{F^k}$ is contained in $X \setminus Y \cup \{0\}$.
\end{theorem}

\begin{theorem}
Let $Y,X$ be spaces of the chain \cref{chain} such that $Y \subsetneq X$.
Then there exist a set $I$ with $\card(I) = \mathfrak{c}$ and a linearly independent family $(F^k)_{k \in I}$ of closed infinite-dimensional vector subspaces of $X$ contained in $X \setminus Y \cup \{0\}$.
\end{theorem}

Let $X$ be a topological vector space of complex-valued sequences and $Y$ be a subset of $X$.
Sufficient conditions for $X \setminus Y$ to be algebraically generic in $Y$ were introduced in \cite{BIG}.
Sufficient conditions for $X \setminus Y$ to be spaceable in $Y$ were also introduced in \cite{BIG}.
We show that in both cases of algebraic genericity and spaceability these conditions are sufficient in order to find uncountably infinitely many linearly independent vector subspaces of $X$ contained in $X \setminus Y \cup \{0\}$ that satisfy the respective property in each case.

Especially in the case of algebraic genericity, we show that these conditions also imply that $X \setminus Y$ is maximal algebraically generic in $X$.
More specifically we show that in our construction of the infinitely many linearly independent dense vector subspaces of $X$ contained in $X \setminus Y \cup \{0\}$ it holds that the subset generated by their union is also contained in $X \setminus Y \cup \{0\}$.
Thus we are also able to prove the following theorem.

\begin{theorem}
\label{thm:intro}
Let $Y,X$ be spaces of the chain \cref{chain} such that $Y \subsetneq X$ and $X \neq \ell^{\infty}$.
Then $X \setminus Y$ is maximal algebraically generic in $X$.
\end{theorem}

We note that Papathanasiou in \cite{PAP} proved that $\ell^{\infty} \setminus c_{0}$ is maximal algebraically generic in $\ell^{\infty}$, thus answering the question of whether \cref{thm:intro} holds in the case $X = \ell^{\infty}$.
It is worth mentioning that algebraic genericity is often referred to as dense lineability in the literature.
More information on the notions of algebraic genericity and spaceability can be found in the book \cite{LINEARITY} and in the expository paper \cite{EXP}.

\section{Preliminaries}

Let $\N = \{1,2,\dots \}$ denote the set of natural numbers and $\N_0 = \N \cup \{0\}$.
If $x : \N_0 \to \C$ is a complex-valued sequence and $n \in \N_0$ we denote by $x(n)$ the value of the sequence $x$ at the number $n$.
We denote by $\C^{\N_0}$ the set of all such complex-valued sequences.
We equip $\C^{\N_0}$ with the product metric obtained by considering each copy of $\C$ with its standard metric.

For every $p \in (0,\infty)$ let
$$\ell^p = \left\{x \in \C^{\N_0} : \sum_{n=0}^{\infty}{\abs{x(n)}^p} < \infty \right\}.$$
We equip $\ell^{p}$ with the metric $d^p$ defined as
$$d^p(x,y) = \left(\sum_{n = 0}^{\infty}{\abs{x(n) - y(n)}^p} \right)^{\frac{1}{p}}$$
for $p \in [1,\infty), x , y \in \ell^p$ and as
$$d^p(x,y) = \sum_{n = 0}^{\infty}{\abs{x(n) - y(n)}^p}$$
for $p \in (0,1), x,y \in \ell^p$.
We also let
$$\ell^{\infty} = \left\{x \in \C^{\N_0} : \sup_{n \in \N_0}{\abs{x(n)}} < \infty \right\}.$$
We equip $\ell^{\infty}$ with the metric $d^{\infty}$ defined as
$$d^{\infty}(x,y) = \sup_{n \in \N_0}{\abs{x(n) - y(n)}}$$
for $x,y \in \ell^{\infty}$.
For every $p \in [0,\infty)$ we equip the space $\bigcap_{q > p}{\ell^q}$ with the metric $\delta^p$ defined as
$$\delta^p(x,y) = \sum_{k = 1}^{\infty}{\frac{1}{2^k}\frac{d^{p_k}(x,y)}{1 + d^{p_k}(x,y)}}$$
where $x,y \in \bigcap_{q > p}{\ell^q}$ and $(p_k)_{k=1}^{\infty}$ is a decreasing sequence in $(0,\infty)$ that converges to $p$.
For the metric $\delta^p$ we make the following remark.
\begin{remark}
Let $(x_m)_{m=1}^{\infty}$ be a sequence in $\bigcap_{q > p}{\ell^q}$ and $x \in \bigcap_{q > p}{\ell^q}$ for some $p \in [0,\infty)$.
We have $x_m \to x$ as $m \to \infty$ with respect to $\delta^p$ if and only if $x_m \to x$ as $m \to \infty$ with respect to $d^q$ for every $q > p$.
\end{remark}
Let
$$c_{0} = \left\{x \in \C^{\N_0} : \lim_{n \to \infty}{x(n)} = 0 \right\}$$
and
$$c_{00} = \left\{x \in \C^{\N_0} : \text{there exists } N \in \N \text{ such that } x(n) = 0 \text{ for all } n \geq N \right\}.$$
We consider both $c_{0}$ and $c_{00}$ as metric subspaces of $\ell^{\infty}$.

Let $A^{\infty}(\D)$ be the set of all holomorphic functions on the open unit disk such that the function and all of its derivatives can be continuously extended on the closed unit disk.
We view $A^{\infty}(\D)$ as a sequence space by identifying every function with the sequence of its Taylor coefficients.
This way we have
$$A^{\infty}(\D) = \left\{x \in \C^{\N_0}: \lim_{n\to \infty}{n^k x(n)} = 0 \text{ for every } k \in \N_0 \right\}.$$
We consider $A^{\infty}(\D)$ as a metric subspace of $\C^{\N_0}$.

Let $H(\D)$ be the set of all holomorphic functions on the open unit disk endowed with the topology of uniform convergence in the compact subsets of the unit disk.
We view $H(\D)$ as a sequence space by identifying every function with the sequence of its Taylor coefficients.
This way we have
$$H(\D) = \left\{x \in \C^{\N_0} : \limsup_{n\to \infty}{\sqrt[n]{\abs{x(n)}}} \leq 1 \right\}.$$

Considered with their respective topologies and the usual addition and scalar multiplication of sequences, each space of the chain \cref{chain} is a metrizable topological vector space.
All the inclusions for the chain \cref{chain} and their strictness have been proven in \cite{BIG,NEST,BERNAL}.

Let $X$ be a space of the chain \cref{chain}, and $(F^k)_{k \in I}$ be a linearly independent family of vector subspaces of $X$.
As $X \subseteq \C^{\N_0}$ and the cardinality of the latter is $\mathfrak{c}$ we have that $\dim(X) \leq \mathfrak{c}$.
Since the family $(F^k)_{k \in I}$ of vector subspaces is linearly independent by selecting a non-zero vector in each subspace $F^k$ for every $k \in I$ we obtain a linearly independent subset of $X$ with cardinality that of the set $I$.
Hence $\card(I) \leq \dim(X)$ and so $\card(I) \leq \mathfrak{c}$.
Therefore we can find at most $\mathfrak{c}$ linearly independent subspaces of $X$.
This motivates the next two definitions.

\begin{definition}
Let $E$ be a topological vector space and $M$ be a subset of $E$.
We say that we have uncountably infinite algebraic genericity for $M$ if and only if there exist a set $I$ with $\card(I) = \mathfrak{c}$ and a linearly independent family $(F^k)_{k \in I}$ of dense vector subspaces of $E$ such that $F^k \subseteq M \cup \{0\}$ for every $k \in I$.
\end{definition}

\begin{definition}
Let $E$ be a topological vector space and $M$ be a subset of $E$.
We say that we have uncountably infinite spaceability for $M$ if and only if there exist a set $I$ with $\card(I) = \mathfrak{c}$ and a linearly independent family $(F^k)_{k \in I}$ of closed infinite-dimensional vector subspaces of $E$ such that $F^k \subseteq M \cup \{0\}$ for every $k \in I$.
\end{definition}

We continue with two well-known lemmata for which we provide sketches of possible proofs.

\begin{lemma}
\label{lem:basic1}
Let $A$ be a countably infinite set.
There exists a family $(A_j)_{j=1}^{\infty}$ of countably infinite pairwise disjoint subsets of $A$, such that $A = \bigcup_{j=1}^{\infty}{A_j}$.
\end{lemma}

\begin{proof}
Since the set $\N \times \N$ is countably infinite there exists a bijection $f: \N \times \N \to A$.
For each $j \in \N$ we define $A_j = \{ f(i,j)  : i \in \N \}$.
The claims follow from the fact that $f$ is a bijection.
\end{proof}

\begin{lemma}
\label{lem:basic2}
Let $A$ be a countably infinite set.
There exists some set $I$ with $\card(I) = \mathfrak{c}$ and a family $(A^k)_{k\in I}$ of countably infinite subsets of $A$ with pairwise finite intersections, such that $A = \bigcup_{k \in I}{A^k}$.
\end{lemma}

\begin{proof}
Consider a bijection of the set $A$ to the set of vertices of a complete binary tree.
Then the set $I$ is the set of infinite branches of the tree and the set $A^k$ is the set of vertices of the branch $k \in I$.
\end{proof}

Finally, concerning notation, if $x \in \C^{\N_0}$ we write $\supp(x) = \{n \in \N_0 : x(n) \neq 0 \}$ for its support and we denote by $\mathds{1}_A$ the characteristic function of the set $A$.

\section{Uncountably infinite algebraic genericity}

We begin by stating and proving the key lemma of this section.

\begin{lemma}
\label{lem:alggen}
Let $X$ be a metrizable topological vector space and $Y$ be a vector subspace of $X$.
We assume the following.
\begin{enumerate}[(i)]
\item It is $c_{00} \subseteq Y \subseteq X \subseteq \C^{\N_0}$.
\item If $A \subseteq \N_0$ is infinite, then there exists $y \in X \setminus Y$ supported in $A$.
\item The space $c_{00}$ is dense in $X$.
\item For every $x \in Y$ and $A \subseteq \N_0$, the sequence $x\mathds{1}_A$ belongs to $Y$.
\end{enumerate}
Then there exists a set $I$ with $\card(I) = \mathfrak{c}$ and a linearly independent family $(F^k)_{k \in I}$ of dense vector subspaces of $X$ such that the vector subspace generated by the set $\bigcup_{k \in I}{F^k}$ is contained in $X \setminus Y \cup \{0\}$.
Furthermore, $X \setminus Y$ is maximal algebraically generic in $X$.
\end{lemma}

\begin{proof}
By assumption \textit{(iii)}, the space $c_{00}$ is dense in $X$ and thus the space $c_{00}\cap (\Q + i\Q)^{\N_0}$ is also dense in $X$.
Let $\{x_j : j \in \N\}$ be an enumeration of $c_{00}\cap (\Q + i\Q)^{\N_0}$.
By \cref{lem:basic2} there exists a set $I$ with $\card(I) = \mathfrak{c}$ and a family  $(A^k)_{k\in I}$ of countably infinite subsets of $\N_0$ with pairwise finite intersections, such that $\N_0 = \bigcup_{k\in I}{A^k}$.
For every $k \in I$, by \cref{lem:basic1}, there exists a family $(A_j^k)_{j=1}^{\infty}$ of countably infinite pairwise disjoint subsets of $A^k$, such that $A^k = \bigcup_{j=1}^{\infty}{A_j^k}$.
We now fix $k \in I$.
For every $j \in \N$, by assumption \textit{(ii)}, there exists $y_j^k \in X \setminus Y$ supported in $A_j^k$.
Since $X$ is a topological vector space, by Birkhoff-Kakutani theorem there exists a metric $d_X$ on $X$ that induces its topology and is translation invariant.
Also, for every $j \in \N$ there exists $c_j^k \in \C \setminus \{0\}$ such that $d_X(c_j^ky_j^k,0)< 1/j.$
For every $j \in \N$ let $f_j^k = x_j + c_j^ky_j^k$.
We define $F^k = \Span\{f_j^k: j \in \N \}$.
As $X$ is a vector space and $f_j^k \in X$ for all $j \in \N$, it follows that $F^k \subseteq X$.
The metric $d_X$ is translation invariant and so $d_X(f_j^k, x_j) < 1/j$ for all $j \in \N$.
Moreover $X$ does not contain any isolated points as a topological vector space and the set $\{x_j : j \in \N\} $ is dense in $X$.
Therefore the set $\{f_j^k: j \in \N\}$ is dense in $X$, which in turn implies that $F^k$ is dense in $X$.

The family $(F^k)_{k \in I}$ of dense vector subspaces of $X$ is linearly independent.
Indeed, let $J$ be a finite subset of $I$ and $v_k \in F^{k}$ for every $k \in J$ be such that $\sum_{k \in J}{v_k} = 0$.
Let us assume by contradiction that there exists some $k_0 \in J$ such that $v_{k_0} \neq 0$.
By definition of $F^{k}$ for every $k \in J$ there exist a natural number $M(k)$ and complex numbers $t_1^{k},\dots, t_{M(k)}^{k}$ such that
$$v_k = \sum_{j = 1}^{M(k)}{t_j^{k}f_j^{k}}.$$
In particular, since $v_{k_0} \neq 0$ we may assume that $t^{k_0}_{M(k_0)} \neq 0$.
Thus
$$0 = \sum_{k \in J}{v_k} = \sum_{k \in J}{\sum_{j = 1}^{M(k)}{t_j^{k}f_j^{k}}} = \sum_{k \in J}{\sum_{j = 1}^{M(k)}{t_j^{k}(c_j^{k}y_j^{k} + x_j)}}.$$
We set $M = \max \{M(k) : k \in J\}$, which exists as the set $J$ is finite.
Since $x_1,\dots,x_M$ are all elements of $c_{00}$ there exists a natural number $N_1$ such that for every $n \geq N_1$ we have $x_j(n) = 0$ for all $j \in \{1,\dots,M\}$.
The set $A^{k_0} \cap A^{k}$ is finite for every $k \in J \setminus \{k_0\}$ and so, as the set $J$ is finite, there exists a natural number $N_2$ such that for all $n \geq N_2$ with $n \in A^{k_0}$ we have $n \not\in A^{k}$ for every $k \in J \setminus \{k_0\}$.
By selection $y^{k_0}_{M(k_0)} \not\in Y$ and by assumption \textit{(i)} it is $c_{00} \subseteq Y$ hence $ y^{k_0}_{M(k_0)} \not\in c_{00}$.
Therefore, there exists a natural number $n_0 \geq \max \{N_1,N_2\}$ such that $y^{k_0}_{M(k_0)}(n_0) \neq 0$.
The sequence $y^{k_0}_{M(k_0)}$ is supported in the set $A^{k_0}_{M(k_0)}$ and so $n_0 \in A^{k_0}_{M(k_0)} \subseteq A^{k_0}$.
By selection of $N_1$ and as $n_0 \geq N_1$ it follows that
$$\left(\sum_{k \in K}{\sum_{j = 1}^{M(k)}{t_j^{k}(c_j^{k}y_j^{k} + x_j)}}\right)(n_0) = \sum_{k \in J}{\sum_{j = 1}^{M(k)}{t_j^{k}c_j^{k}y_j^{k}(n_0)}}.$$
By selection of $n_0,N_2$, the fact that $\supp(y_j^k) \subseteq A^k$ for every $k \in I$ and $j \in \N$ and as $n_0 \geq N_2$, it follows that $n_0 \in \supp\left(y^{k_0}_{M(k_0)}\right)$ and $n_0 \not\in \supp(y^k_j)$ for every $k \in J\setminus \{k_0\}$ and $j \in \N$.
Hence
$$\sum_{k \in J}{\sum_{j = 1}^{M(k)}{t_j^{k}c_j^{k}y_j^{k}(n_0)}} = \sum_{j = 1}^{M(k_0)}{t_j^{k_0}c_j^{k_0}y_j^{k_0}(n_0)}.$$
Since $\supp(y_j^{k_0}) \subseteq A^{k_0}_j$ for every $j \in \N$ with the sets $(A^{k_0}_j)_{j = 1}^{\infty}$ being pairwise disjoint and $n_0 \in A_{M(k_0)}^{k_0}$ we obtain that
$$\sum_{j = 1}^{M(k_0)}{t_j^{k_0}c_j^{k_0}y_j^{k_0}(n_0)} = t_{M(k_0)}^{k_0}c_{M(k_0)}^{k_0}y_{M(k_0)}^{k_0}(n_0).$$
Combining the last three equalities with our assumption that
$$\sum_{k \in J}{\sum_{j = 1}^{M(k)}{t_j^{k}(c_j^{k}y_j^{k} + x_j)}} = 0$$
we deduce that
$$t_{M(k_0)}^{k_0}c_{M(k_0)}^{k_0}y_{M(k_0)}^{k_0}(n_0) = 0.$$
But, $t_{M(k_0)}^{k_0} \neq 0$ by our assumption and $c_{M(k_0)}^{k_0} \neq 0$ by its selection.
Thus $y_{M(k_0)}^{k_0}(n_0) = 0$.
This is a contradiction as the natural number $n_0$ was chosen so that  $y_{M(k_0)}^{k_0}(n_0) \neq 0$.

Let $F$ be the vector subspace of $X$ generated by the set $\bigcup_{k \in I}{F^k}$, that is,
$$F = \Span\left( \bigcup_{k \in I}{F^k} \right).$$
We show that $F \subseteq X \setminus Y \cup \{0\}$.
Let $v \in F \setminus \{0\}$ be arbitrary.
Then there exists a finite subset $J$ of $I$ and $v_k \in F^{k}$ for every $k \in J$ such that $v = \sum_{k\in J}{v_k}$, because every $F^k$ for $k \in I$ is a vector space.
By definition of $F^{k}$ for every $k \in J$ there exist a natural number $M(k)$ and complex numbers $t_1^{k},\dots, t_{M(k)}^{k}$ such that
$$v_k = \sum_{j = 1}^{M(k)}{t_j^{k}f_j^{k}}.$$
Thus
$$v = \sum_{k \in J}{\sum_{j = 1}^{M(k)}{t_j^{k}f_j^{k}}}$$
and so
$$v = \sum_{k \in J}{\sum_{j = 1}^{M(k)}{t_j^{k}(c_j^{k}y_j^{k} + x_j)}}.$$
As $v \neq 0$ there has to be some $k_0 \in J$ such that $v_{k_0} \neq 0$ and so we may assume that $t_{M(k_0)}^{k_0} \neq 0$.
We set $M = \max \{M(k) : k \in J\}$, which exists as the set $J$ is finite.
As $x_1,\dots,x_M$ are elements of $c_{00}$ there exists a natural number $N_1$ such that for all $n \geq N_1$ we have $x_j(n) = 0$ for every $j \in \{1,\dots,M\}$.
The set $A^{k} \cap A^{k'}$ is finite for every $k, k' \in J$ with $k \neq k'$ and so, as the set $J$ is finite, there exists a natural number $N_2$ such that $A^{k} \cap A^{k'} \cap [N,\infty) = \varnothing$ for every $k,k' \in J$ with $k \neq k'$.
We set $N = \max\{N_1, N_2\}$ and
$$B = A^{k_0}_{M(k_0)} \cap [N, \infty).$$
Suppose by contradiction that $v \in Y$.
Then, by assumption \textit{(iv)} we would have $v\mathds{1}_{B} \in Y$.
However, since $B \subseteq [N, \infty)$ and $N \geq N_1$ it follows that
$$v\mathds{1}_{B} = \sum_{k \in J}{\sum_{j = 1}^{M(k)}{t_j^{k}c_j^{k}y_j^{k}\mathds{1}_{B}}}.$$
For every $k \in J$ and every $j \in \{1,\dots,M(k)\}$ it is $\supp(y^{k}_j) \subseteq A^{k}_j \subseteq A^{k}$.
We claim that for every $k,k' \in \{1\dots,m\}$ and all $j \in \{1,\dots,M(k)\}$ and $j' \in \{1,\dots,M(k')\}$ such that $(k,j) \neq (k',j')$ we have that $A^{k}_{j} \cap A^{k'}_{j'} \cap [N,\infty) = \varnothing$.
Indeed, if $k \neq k'$ this follow immediately by the selection of $N_2$ and the inclusion $$A^{k}_{j} \cap A^{k'}_{j'} \cap [N,\infty) \subseteq A^{k} \cap A^{k'} \cap [N_2,\infty)$$
which holds as $N \geq N_2, A^{k}_j \subseteq A^k$ and $A^{k'}_{j'} \subseteq A^{k'}$.
If $k = k'$ and $j \neq j'$ then our claim follows from the fact that the sets $(A^k_{h})_{h=1}^{\infty}$ are pairwise disjoint.
Considering how the set $B$ was defined we deduce that $\supp(y^{k}_j) \cap B = \varnothing$ for every $k \in J$ and all $j \in \{1,\dots,M(k)\}$ such that $(k,j) \neq (k_0, M(k_0))$.
Hence
$$\sum_{k \in J}{\sum_{j = 1}^{M(k)}{t_j^{k}c_j^{k}y_j^{k}\mathds{1}_{B}}} = t_{M(k_0)}^{k_0}c_{M(k_0)}^{k_0}y_{M(k_0)}^{k_0}\mathds{1}_{B}.$$
It is $\supp\left(y_{M(k_0)}^{k_0}\right) \subseteq A_{M(k_0)}^{k_0}$ and so
$$t_{M(k_0)}^{k_0}c_{M(k_0)}^{k_0}y_{M(k_0)}^{k_0}\mathds{1}_{B} = t_{M(k_0)}^{k_0}c_{M(k_0)}^{k_0}y_{M(k_0)}^{k_0}\mathds{1}_{[N,\infty)}.$$
Therefore we have shown that
$$v\mathds{1}_{B} = t_{M(k_0)}^{k_0}c_{M(k_0)}^{k_0}y_{M(k_0)}^{k_0}\mathds{1}_{[N,\infty)} \in Y.$$
We observe that
$$t_{M(k_0)}^{k_0}c_{M(k_0)}^{k_0}y_{M(k_0)}^{k_0} = t_{M(k_0)}^{k_0}c_{M(k_0)}^{k_0}y_{M(k_0)}^{k_0}\mathds{1}_{[0,N)} + t_{M(k_0)}^{k_0}c_{M(k_0)}^{k_0}y_{M(k_0)}^{k_0}\mathds{1}_{[N,\infty)}$$
and the sequence $t_{M(k_0)}^{k_0}c_{M(k_0)}^{k_0}y_{M(k_0)}^{k_0}\mathds{1}_{[0,N)} $ belongs in $Y$ as it belongs to $c_{00}$ and by assumption \textit{(i)} we have $c_{00} \subseteq Y$.
Also the complex numbers $c_{M(k_0)}^{k_0}$ and $t_{M(k_0)}^{k_0}$ are both non-zero.
As $Y$ is a vector subspace of $X$ it follows that $y_{M(k_0)}^{k_0} \in Y$, which contradicts the selection of $y_{M(k_0)}^{k_0}$.

Finally, $F$ is a dense vector subspace of $X$, as it contains the dense subspaces $F^k$ for $k \in I$, that is contained in $X \setminus Y \cup \{0\}$.
Since the family $(F^k)_{k \in I}$ of vector subspaces of $F$ is linearly independent by selecting a non-zero vector from each subspace $F^k$ we obtain a linearly independent subset of $F$ with cardinality that of the set $I$.
Therefore $\dim(F) \geq \card(I)$.
On the other hand, $F$ is a vector subspace of $X$ and so $\dim(F) \leq \dim(X)$.
By assumption \textit{(i)} it is $X \subseteq \C^{\N_0}$ which implies that $\dim(X) \leq \card(\C^{\N_0})$.
We conclude that
$$\mathfrak{c} = \card(I) \leq \dim(F) \leq \dim(X) \leq \card(\C^{\N_0}) = \mathfrak{c}.$$
It follows that $\dim(F) = \dim(X)$.
Hence $X \setminus Y$ is maximal algebraically generic in $X$.
\end{proof}

The proposition that follows ensures that we can apply the previous lemma to the chain in which we are interested and it has already been proven as Proposition 4.4 in \cite{BIG}.

\begin{proposition}
\label{prop:agassumptions}
If $Y,X$ are spaces of the chain \cref{chain} such that $Y \subsetneq X$ and $X \neq \ell^{\infty}$, then the assumptions of \cref{lem:alggen} are satisfied.
\end{proposition}

\cref{prop:agassumptions} and \cref{lem:alggen} lead to the first main result of this section.

\begin{theorem}
Let $Y,X$ be spaces of the chain \cref{chain} such that $Y \subsetneq X$ and $X \neq \ell^{\infty}$.
Then there exists a set $I$ with $\card(I) = \mathfrak{c}$ and a linearly independent family $(F^k)_{k \in I}$ of dense vector subspaces of $X$ such that the vector subspace generated by the set $\bigcup_{k \in I}{F^k}$ is contained in $X \setminus Y \cup \{0\}$.
In particular, we have uncountably infinite algebraic genericity for $X\setminus Y$ in $X$.
\end{theorem}

The next theorem is also a consequence of \cref{prop:agassumptions} and \cref{lem:alggen}.

\begin{theorem}
\label{thm:second}
Let $Y,X$ be spaces of the chain \cref{chain} such that $Y \subsetneq X$ and $X \neq \ell^{\infty}$.
Then $X \setminus Y$ is maximal algebraically generic in $X$.
\end{theorem}

Papathanasiou in \cite{PAP} demonstrated that $\ell^{\infty} \setminus c_{0}$ is maximal algebraically generic in $\ell^{\infty}$.
That is, there exists a dense subspace $F$ of $\ell^{\infty}$ with $\dim(F) = \dim(\ell^{\infty})$ such that $F \subseteq \ell^{\infty} \setminus c_{0} \cup \{0\}$.
If $Y$ is any space of the chain \cref{chain} such that $Y \subsetneq \ell^{\infty}$ then $Y \subseteq c_{0}$ and so $F \subseteq \ell^{\infty} \setminus Y \cup \{0\}$.
Thus $\ell^{\infty} \setminus Y$ is maximal algebraically generic in $\ell^{\infty}$.
Combining this observation with \cref{thm:second} proves the following theorem.

\begin{theorem}
Let $Y,X$ be spaces of the chain \cref{chain} such that $Y \subsetneq X$.
Then $X \setminus Y$ is maximal algebraically generic in $X$.
\end{theorem}

\section{Uncountably infinite spaceability}

We begin by stating and proving the key lemma of this section.

\begin{lemma}
\label{lem:space}
Let $X$ be a topological vector space and $Y$ be a subset of $X$ closed under scalar multiplication.
We assume the following.
\begin{enumerate}[(i)]
\item It is $c_{00} \subseteq Y \subseteq X \subseteq \C^{\N_0}$.
\item If $A \subseteq \N_0$ is infinite, then there exists $y \in X \setminus Y$ supported in $A$.
\item Convergence in $X$ implies pointwise convergence.
\item For every $x \in Y$ and $A \subseteq \N_0$, the sequence $x\mathds{1}_A$ belongs to $Y$.
\end{enumerate}
Then we have uncountably infinite spaceability for $X \setminus Y$ in $X$.
\end{lemma}

\begin{proof}
By \cref{lem:basic2} there exists a set $I$ with $\card(I) = \mathfrak{c}$ and a family  $(A^k)_{k\in I}$ of countably infinite subsets of $\N_0$ with pairwise finite intersections such that $\N_0 = \bigcup_{k\in I}{A^k}$.
For every $k \in I$, by \cref{lem:basic1}, there exists a family $(A_j^k)_{j=1}^{\infty}$ of countably infinite pairwise disjoint subsets of $A^k$ such that $A^k = \bigcup_{j=1}^{\infty}{A_j^k}$.
We now fix $k \in I$.
For every $j \in \N$, by assumption \textit{(ii)}, there exists $y_j^k \in X \setminus Y$ supported in $A_j^k$.
We define $F^k = \cl_X(\Span \{y_j^k: j \in \N \} )$.
Then $F^k$ is a closed linear subspace of $X$.
Furthermore $F^k$ is infinite-dimensional because the set $\{y_j^k : j \in \N\}$ is linearly independent as the supports of its elements are pairwise disjoint.

Next, we show that $F^k \subseteq X \setminus Y \cup \{0\}$.
Indeed, let $f \in F^k$ be arbitrary such that $f \neq 0$.
Clearly $f \in X$ so we must show that $f \not\in Y$.
Since $f \in F^k$ there exists a sequence $(f_h)_{h=1}^{\infty}$ in $\Span\{y^k_j: j \in \N \}$ such that $f_h \to f$ in $X$.
For each $h \in \N$ we can write $$f_h = \sum_{j=1}^{\infty}{c_{h,j} y^k_j}$$
with the set $\{ j \in \N : c_{h,j} \neq 0 \}$ being finite.
This implies that
$$ \supp(f_h) \subseteq \bigcup_{j=1}^{\infty}{\supp(y^k_j)} \subseteq \bigcup_{j=1}^{\infty}{A^k_j} = A^k$$
for all $h \in \N$.
By assumption \textit{(iii)} and the convergence $f_h \to f$ in $X$ it follows that $f_h(n) \to f(n)$ for all $n \in \N_0$.
Thus
$$\supp(f) \subseteq \bigcup_{h=1}^{\infty}{\supp(f_h)} \subseteq \bigcup_{j=1}^{\infty}{A^k_j} = A^k,$$
because if for some $n \in \N_0$ we have for all $h \in \N$ that $f_h(n) = 0$, then $f(n) = 0$ as it is $f_h(n) \to f(n)$.
Now, because $f \neq 0$ there exists $n_0 \in \N_0$ such that $f(n_0) \neq 0$ and so $n_0 \in \bigcup_{j=1}^{\infty}{A^k_j}$.
Hence there exists some $j_0 \in \N$ such that $n_0 \in A^k_{j_0}$.
Since each sequence $y_j^k$ is supported in the set $A_j^k$ with the sets $(A_j^k)_{j=1}^{\infty}$ being pairwise disjoint, for all $n \in A_{j_0}^k$ and all $h \in \N$ we have that $f_h(n) = c_{h,j_0}y_{j_0}^k(n)$.
If it was $y_{j_0}(n_0) = 0$ then we would have that $f_h(n_0) = 0$ for all $h \in \N$ which is not possible as
$$n_0 \in \supp(f) \subseteq \bigcup_{h=1}^{\infty}{\supp(f_h)}.$$
Therefore $y_{j_0}^k(n_0) \neq 0$.
Observe that since $f(n_0) = \lim_{h\to \infty}{f_h(n_0)}$ and $f(n_0) \neq 0$ we have
$$\lim_{h \to \infty}{\frac{f_h(n_0)}{y_{j_0}^k(n_0)}} = \frac{f(n_0)}{y_{j_0}^k(n_0)} \neq 0.$$
On the other hand, for all $h \in \N$ it is $f_h(n_0) = c_{h,j_0}y_{j_0}^k(n_0)$ and thus
$$\lim_{h \to \infty}{c_{h,j_0}} = \lim_{h \to \infty}{\frac{f_h(n_0)}{y_{j_0}^k(n_0)}} = \frac{f(n_0)}{y_{j_0}^k(n_0)} \neq 0.$$
Let $c_{j_0} = \lim_{h \to \infty}{c_{h,j_0}}$.
Then $c_{j_0} \neq 0$ and for all $n \in A_{j_0}^k$ we have
$$f(n) = \lim_{h \to \infty}{f_h(n)} = \lim_{h\to \infty}{c_{h,j_0}y_{j_0}^k(n)} = c_{j_0}y_{j_0}^k(n).$$
This means that $f\mathds{1}_{A_{j_0}^k} = c_{j_0}y_{j_0}^k$.
Thus $f\mathds{1}_{A_{j_0}^k} \not\in Y$.
Indeed, if it were $f\mathds{1}_{A_{j_0}^k} \in Y$ that would imply that $c_{j_0}y_{j_0}^k \in Y$ and $Y$ being closed under scalar multiplication while $c_{j_0} \neq 0$ would in turn imply that $y_{j_0}^k \in Y$ which is not possible since we chose the sequence $y_{j_0}^k$ such that $y_{j_0}^k \in X \setminus Y$.
By the contrapositive of assumption \textit{(iv)} it follows that $f \not\in Y$ as wanted.

Finally, it remains to be shown that the family $(F^k)_{k \in I}$ of closed infinite-dimensional vector subspaces of $X$ contained in $X \setminus Y$ is linearly independent.
To this end, let $J$ be a finite subset of $I$ and $v_k \in F^{k}$ for every $k \in I$ be such that $\sum_{k \in J}{v_k}= 0$.
Suppose by contradiction that there exists some $k_0 \in J$ such that $v_{k_0} \neq 0$.
The set $A^{k} \cap A^{k_0}$ is finite for every $k \in J \setminus \{k_0\}$ and so, as the set $J$ is finite, there exists some natural number $N$ such that for every  natural number $n \geq N$ with $n \in A^{k_0}$ we have $n \not\in A^{k}$ for all $k \in J \setminus \{k_0\}$.
We have shown in the previous paragraph that if $f \in F^k$ for some $k \in I$, then $\supp(f) \subseteq A^k$.
Therefore $\supp(v_k) \subseteq A^{k}$ for every $k \in J.$
It follows that if for every natural number $n \geq N$ with $n \in \supp(v_{k_0})$ we have $n \not\in \supp(v_k)$ for all $k \in J \setminus \{k_0\}$.
Equivalently, for every  natural number $n \geq N$ with $v_{k_0}(n) \neq 0$ we have $v_k(n) = 0$ for all $k \in J \setminus \{k_0\}$.
Because $v_{k_0} \neq 0$ and $F^{k_0} \subseteq X \setminus Y \cup \{0\}$ we deduce that $v_{k_0} \not\in Y$.
By assumption \textit{(i)} it is $c_{00}\subseteq Y$ and so $v_{k_0} \not\in c_{00}$.
Thus there exists a natural number $n_0 \geq N$ such that $v_{k_0}(n_0) \neq 0$.
But then our observation above and the assumption that $\sum_{k \in J}{v_j} = 0$ imply that
$$0 = \left(\sum_{k \in J}{v_k}\right)(n_0) = \sum_{k \in J}{v_k(n_0)} = v_{k_0}(n_0),$$
which is absurd.

\end{proof}

The proposition that follows ensures that we can apply the previous lemma to the chain in which we are interested.

\begin{proposition}
\label{prop:spassumptions}
If $Y,X$ are spaces of the chain \cref{chain} such that $Y \subsetneq X$, then the assumptions of \cref{lem:space} are satisfied.
\end{proposition}

\begin{proof}
Notice that assumptions \textit{(i),(ii)} and \textit{(iv)} are the same with \cref{lem:alggen} and therefore we have already seen that they are satisfied if $X \neq \ell^{\infty}$.
On the other hand, if $X = \ell^{\infty}$ assumptions \textit{(i),(ii)} and \textit{(iv)} are easily verified.
Concerning assumption \textit{(iii)}, for the cases $X = H(\D)$ and $X = A^{\infty}(\D)$ one should look at Propositions 2.5 and 2.6 in \cite{BIG} respectively.
The remaining cases are standard.
\end{proof}

Combining \cref{lem:space} with \cref{prop:spassumptions} we obtain the main result of this section.

\begin{theorem}
Let $Y,X$ be spaces of the chain \cref{chain} such that $Y \subsetneq X$.
Then we have uncountably infinite spaceability for $X \setminus Y$ in $X$.
\end{theorem}

\section*{\fontsize{11}{15}\selectfont Acknowledgements}
The author would like to express his gratitude towards I. Deliyanni for her valuable suggestions.
The author would also like to thank V. Nestoridis for his interest in this work and his advice concerning this paper.

\printbibliography
\bigskip
\noindent
C. A. Konidas\\
National and Kapodistrian University of Athens\\
Department of Mathematics\\
e-mail address: \href{mailto:xkonidas@gmail.com}{\tt xkonidas@gmail.com}

\end{document}